\documentclass[12pt]{article}

\usepackage[margin=1in]{geometry}
\usepackage{graphicx}    
\usepackage{amsmath}    
\usepackage{amsfonts}    
\usepackage{amsthm}    
\usepackage{amssymb}
\usepackage{amscd}
\usepackage{extarrows}
\usepackage[shortlabels]{enumitem}
\usepackage{secdot}
\usepackage{tikz}
\usepackage{sectsty}

\sectionfont{\center \normalsize}
\subsectionfont{\center \normalsize}
\subsubsectionfont{\normalsize}
\paragraphfont{\normalsize}

\sectiondot{subsection}

\newtheorem{thm}{Theorem}[section]

\newtheorem{lem}[thm]{Lemma}
\newtheorem{prop}[thm]{Proposition}

\newtheorem{ques}[thm]{Question}

\theoremstyle{definition}

\theoremstyle{remark}
\newtheorem{rem}[thm]{Remark}

\theoremstyle{remark}

\theoremstyle{remark}

\theoremstyle{remark}

\theoremstyle{remark}

\theoremstyle{remark}
\newtheorem{ex}[thm]{Example}

\DeclareMathOperator{\Br}{Br}
\DeclareMathOperator{\C}{\mathbb{C}}

\DeclareMathOperator{\cores}{cor}

\DeclareMathOperator{\G}{\mathbb{G}}

\DeclareMathOperator{\ind}{ind}

\DeclareMathOperator{\Q}{\mathbb{Q}}

\DeclareMathOperator{\Z}{\mathbb{Z}}

\makeatletter
\newcommand{\etale}{\'etal\@ifstar{\'e}{e\space}}
\makeatother
\makeatletter
\newcommand{\CT}{Colliot-Th\'el\`en\@ifstar{\'e}{e}}
\makeatother

\makeatletter
\newcommand{\Veis}{Ve\u\i sfe\u\i ler}
\makeatother


\newcommand*{\TitleFont}{%
  \usefont{\encodingdefault}{\rmdefault}{b}{n}%
  \fontsize{14}{20}%
  \selectfont}

\newcommand*{\AddFont}{%
  \usefont{\encodingdefault}{\rmdefault}{}{n}%
  \fontsize{10}{20}%
  \selectfont}

\begin{document}

\nocite{*}

\title{\TitleFont \textbf{ TOTARO'S QUESTION ON ZERO-CYCLES ON TORSORS}}
\author{R. GORDON-SARNEY AND V. SURESH}
\date{\AddFont DEPARTMENT OF MATHEMATICS \& COMPUTER SCIENCE \\ EMORY UNIVERSITY, ATLANTA, GA 30322 USA}

\maketitle

\begin{abstract}
Let $G$ be a smooth connected linear algebraic group and $X$ be a $G$-torsor.
Totaro asked: if $X$ admits a zero-cycle of degree $d \geq 1$, then does $X$ have a closed \etale point of degree dividing $d$?
While the literature contains affirmative answers in some special cases, we give an example to 
show that the answer is negative in general.
\end{abstract}
\vspace{0.15cm}
\section{Introduction}
\indent
\indent
One approach to understanding the rational points on a variety $X$ over a field $k$ is to study its group of zero-cycles, denoted $Z_0(X)$.
Every rational point on $X$ can be viewed as a zero-cycle of degree 1, where the degree homomorphism $\deg : Z_0(X) \to \Z$ associates each closed point $x \in X$ to the degree of its residue field $[k(x):k]$.
It is natural to ask about the converse: if a variety $X$ admits a zero-cycle of degree 1, does $X$ have a rational point?

The question was originally raised by Serre in the `60s in the case of principal homogeneous spaces (or torsors) under smooth connected linear algebraic groups over fields.
\\\\
\textbf{Serre's Question}. 
Let $G$ be a smooth connected linear algebraic group over a field $k$.
If a $G$-torsor $X$ admits a zero-cycle of degree 1, does $X$ have a rational point?
\\\\
The positive answer to Serre's question for torsors under projective general linear groups is a classical theorem on central simple algebras.
Springer's theorem on quadratic forms answers the question in the affirmative for torsors under orthogonal groups \cite{springer52}.
Bayer--Lenstra settled the question for torsors under unitary groups \cite{bayerlenstra90}.
Sansuc gave an affirmative answer to the question for torsors under any smooth connected linear algebraic group defined over a number field \cite{sansuc81}.
Affirmative answers are known in many other special cases (cf. \cite{niv16}, \cite{black11a}, \cite{black11b}), though in general, Serre's question is still open. 

% Black proved the result for groups over fields of virtual cohomological dimension $\leq 2$ whose simply connected covers satisfy a Hasse principle \cite{black11a}.
% Additionally, Black gave a positive answer over fields of characteristic $\neq 2$ for simply connected or adjoint semisimple groups without any $E_8$ factors such that every exceptional simple factor of type other than $G_2$ is quasisplit \cite{black11b}.
% Bhaskhar showed the same over fields of characteristic $\neq 2$ for smooth connected reductive groups whose Dynkin diagrams have connected components only of types $A_n$, $B_n$, and $C_n$ \cite{niv16}.

In the setting of general varieties, which are not necessarily torsors, the literature is rich with striking counterexamples.
The classical Weil estimates show that any curve of genus $\geq 2$ over a finite field without a point nevertheless admits a zero-cycle of degree 1.
\CT--Coray produced a conic bundle over the $p$-adic projective line--a rational variety--admitting a zero-cycle of degree 1 but no rational points \cite{colliocoray79}.
Florence constructed an affine homogeneous space under a smooth connected linear algebraic group with finite stabilizers over $\C((x))((y))$ or a local or global field with this same property \cite{florence04}.
Parimala gave as a counterexample a projective homogeneous space under a smooth connected linear algebraic group over a $p$-adic Laurent series field \cite{pari05}, settling a long-standing conjecture of \Veis\ in the negative \cite{veis69}.
Motivated in part by the classical result on central simple algebras, Totaro posed the following generalization of Serre's question in 2004:
\\\\
\textbf{Totaro's Question} (\cite{totaro04}). 
Let $G$ be a smooth connected linear algebraic group over a field $k$.
If a $G$-torsor $X$ admits a zero-cycle of degree $d \geq 1$, does $X$ have a closed \etale point of degree dividing $d$?
\\

Affirmative answers have been far rarer and far more specialized for Totaro's question than for Serre's question.
In the paper where the question was originally posed, Totaro handled the cases of split simply connected groups of type $G_2$, $F_4$, and $E_6$ with a partial result on $E_7$.
Garibaldi--Hoffman extended Totaro's results to all groups of type $G_2$, reduced of type $F_4$, and simply connected of type ${}^1 E_{6,6}^0$ and ${}^1 E_{6,2}^{28}$ \cite{garihoff06}.
Black--Parimala settled Totaro's question in the affirmative for simply connected, semisimple groups of rank $\leq 2$ over fields of characteristic $\neq 2$ \cite{blackpari14}.
Recently, the first author gave an affirmative answer to Totaro's question for algebraic tori of rank $\leq 2$ over arbitrary fields \cite{totori} and absolutely simple adjoint groups of classical types $A_1$ and $A_{2n}$ over fields of characteristic $\neq 2$ \cite{totaro_adjoint}.

In this paper, we construct examples to show that the answer to Totaro's question is negative in general. 
In Section 2, we obtain the following (see \ref{rank3-semisimple}, \ref{colliot}):

\begin{thm} 
Let $k$ be a global field of characteristic not equal to 2 and $K$ a complete discretely valued field with residue field $k$. 
Then for every integer $n \geq 3$, there exist a connected semisimple linear algebraic group $G$ 
of rank $n$ over $k$ and a $G$-torsor $X$ such that $X$ admits a zero-cycle of degree $2$ but has no closed point of degree 1 or 2.
\end{thm}

In the light of the first author's work on tori of rank at most 2, the following consequence of the above result (see \ref{rank3-torus}, \ref{colliot}) is
interesting.
 
\begin{thm} 
Let $k$ be a global field of characteristic not equal to 2 and $K$ a complete discretely valued field with residue field $k$. 
Then for every integer $n \geq 3$, there exist a torus $T$ of rank $n$ over $K$ and a $T$-torsor $X$ such that 
$X$ admits a zero-cycle of degree $2$ but has no closed point of degree 1 or 2.
\end{thm}

By the result of Sansuc \cite{sansuc81}, Serre's question has an affirmative answer for groups over global and local fields.
In Section 3, we prove the following (see \ref{example-lf}, \ref{example-gf}):

\begin{thm} 
Let $k$ be a global field or a local field. 
Then there exist examples of  semisimple groups and tori of rank 8 over $k$  which  admit  torsors with  
 zero-cycles of degree $2$ but have no closed points of degree 1 or 2.
\end{thm}

The groups in each example where Totaro's question has a negative answer are not absolutely simple. 
In view of this, we ask the following question.

\begin{ques} 
Let $k$ be field, let $G$ a smooth absolutely simple linear algebraic group over $k$, and let $X$ be a $G$-torsor. 
If $X$ admits a zero-cycle of degree $d \geq 1$, does $X$ admit a closed \etale point of degree dividing $d$?
\end{ques}

\section{Examples of Rank $p$ Groups }

In this section, for every odd prime $p$, we give examples of rank $p$ semisimple groups and rank $p$ tori for which Totaro's question has a negative answer. 
In general, the corestriction from a finite extension of a quaternion algebra may not be a quaternion algebra. 
We begin by constructing such algebras explicitly. 
 
\begin{lem}
\label{cor}
Let $k$ be a global field of characteristic $\neq 2$ and $p$ an odd prime. Let $\ell/k$ be a separable 
field extension of  degree $p$. 
Then there exist a quaternion division algebra $Q$ over $k$ and $\lambda \in \ell^*$ such that
\begin{enumerate}
\item $Q \otimes_k \ell(\sqrt{\lambda})$ is split,
\item $N_{\ell/k}(\lambda) \not\in k^{*2}$, and
\item $Q \otimes_k k(\sqrt{N_{\ell/k}(\lambda)})$ is division. 
\end{enumerate}
\end{lem}

\begin{proof} For  a place $\nu$ of a global field $k$, let $k_{\nu}$ denote the completion of $k$ at $\nu$. 
Let $\nu$ and $\nu'$ be two distinct places of $k$ which split in $\ell$ (cf. \cite[Ch.VII, Theorem 13.4]{Neu}). 
Let $Q$ be the division algebra over $k$ such that $Q \otimes_k k_{\nu}$ and $Q \otimes_k k_{\nu'}$ are division and $Q \otimes_k k_\omega$ is split for all places $\omega$ of $k$ not equal to $ \nu$ and $\nu'$ (cf. \cite[p. 196]{CF}).
Let $\nu_1, \nu_2, \ldots, \nu_p$ and $\nu'_1, \nu'_2, \ldots, \nu'_p$ be the places of $\ell$ lying over $\nu$ and $\nu'$, respectively. 
Let $\pi, \pi' \in k$ be parameters at $\nu$ and $\nu'$, respectively. 
Let $\theta, \theta' \in k$ be non-square units at $\nu$ and $\nu'$, respectively. 
Let $\nu''$ be a place of $k$ with $\ell \otimes k_{\nu''}$ a field, and let $\tilde{\nu}''$ be the unique extension of $\nu''$ to $\ell$. 
Since $[\ell: k]$ is odd, there exists $\theta'' \in \ell_{\tilde{\nu}''} = \ell \otimes k_{\nu''}$ with $N_{\ell_{\tilde{\nu}''}/k_{\nu''}}(\theta'')$ not a square in $k_{\nu''}$. 

By weak approximation, let $\lambda \in \ell^*$ such that

\indent $\bullet$ is close to $\pi$ (resp. $\pi'$) at $\nu_1$ (resp. $\nu'_1$), \\
\indent $\bullet$ is close to $\theta$ (resp. $\theta'$) at $\nu_2, \ldots, \nu_{p-1}$ (resp. $\nu_2', \ldots, \nu_{p-1}'$), \\
\indent $\bullet$ is close to $\pi^{-1}\theta^{-(p-2)}$ (resp. $\pi^{'-1}\theta^{'-(p-2)}$) at $\nu_p$ (resp. $\nu_p')$, and \\
\indent $\bullet$ is close to $\theta''$ at $\tilde{\nu}''$.

We now show that $Q$ and $\lambda$ have the required properties. 
By the choice of $\lambda$, $\lambda$ is not a square in $\ell_{\nu_i}$ for all $i$.
In particular, $Q \otimes_k \ell_{\nu_i}(\sqrt{\lambda})$ is a split algebra (cf. \cite[p. 131, Corollary 1]{CF}). 
Similarly, $Q \otimes_k \ell_{\nu'_i}(\sqrt{\lambda})$ is a split algebra.
Since $Q \otimes_k k_\omega$ is a split algebra for all $\omega$ not equal to $\nu$ and $\nu'$, $Q \otimes \ell(\sqrt{\lambda})$ is a split algebra (cf. 
\cite[p. 187, Corollary 9.8]{CF}). 
Since $\lambda$ is close to $\theta''$ at $\tilde{\nu}''$ and $N_{\ell_{\tilde{\nu}''}/k_{\nu''}}(\theta'')$ is a not a square in $k_{\nu''}$, $N_{\ell/k}(\lambda)$ is not a square in $k$. 
By the choice of $\lambda$, $N_{\ell/k}(\lambda)$ is close to $1$ at $\nu$, hence $\sqrt{N_{\ell/k}(\lambda)} \in k_\nu$. 
Since $Q \otimes_k k_\nu$ is division, $Q \otimes_k k(\sqrt{N_{\ell/k}(\lambda)})$ is division. 
Thus $Q$ and $\lambda$ have the required properties. 
\end{proof}

\begin{prop} 
\label{fields} Let $k$ be a global field of characteristic $\neq 2$ and $p$ an odd prime. Let $\ell/k$ be a separable 
field extension of   $p$.
Let $K$ be a complete discretely  valued field with residue field $k$ and $L/K$ the unramified extension of degree $p$ with residue field 
$\ell$.  Then there exists  a quaternion division algebra $D$ over $L$ such that $\ind(\cores_{L/K}(D) ) = 4$. 
\end{prop}

\begin{proof} 
Let $R$ be the ring of integers in $K$.
Let $\pi \in R$ be a parameter. 
Let $p$ be an odd prime and $\ell/k$ a field extension of degree $p$. 
Let $Q$ be a quaternion division algebra over $k$ and $\lambda \in \ell^*$ be as in (\ref{cor}). 
Let $\tilde{D}_0$ be the quaternion algebra over $R$ with $\tilde{D}_0 \otimes_R R/(\pi) \simeq Q$ (cf.  \cite{cipolla77}) and $D_0 = \tilde{D}_0 \otimes_R K$. 
Let $u \in L$ be a unit in the ring of integers of $L$ which maps to $\lambda \in \ell$. 
Since $Q \otimes_k \ell(\sqrt{\lambda})$ is a split algebra, $D_0 \otimes_K L(\sqrt{u})$ is a split algebra (cf. \cite{cipolla77}).
In particular, $D_0 \otimes_K L = (u, a)$ for some $a \in L$ (cf. \cite[Proposition 1.2.3]{gisz06}). Let $D = (u, a\pi)$.
Then $D = D_0 \otimes_K L + (u, \pi)$. 

We now show that ind(cor$_{L/K}(D)) = 4$. 
Since cor$_{L/K}(D_0 \otimes_K L) = [L : k]D_ 0 = pD_0$ (cf. \cite[Proposition 4.2.10]{gisz06}) with $p$ odd and cor$_{L/K}(u, \pi) = ( N_{L/K}(u), \pi) $ (cf. \cite[Section 4.7, Proposition 9.(iv) ]{CF}), we have cor$_{L/K}(D) = D_0 + (N_{L/K}(u), \pi)$.
Since $D_0$ is unramified on $R$, $N_{L/K}(u)$ is unit in $R$ and $D_0 \otimes K(\sqrt{N_{L/K}(u)})$ is a division algebra, by (\cite[Proposition 1(3) ]{feinschacher95}), we have $${\rm ind}(D_0 \otimes (N_{L/K}(u), \pi)) = {\rm ind}(D \otimes_K K(\sqrt{N_{L/K}(u)})) [K(\sqrt{N_{L/K}(u)}) : K] = 4.$$
\end{proof}

\begin{prop}
\label{split}
Let $K$ be a field and $L/K$ a finite separable  extension of degree $p$ a prime. 
Let $A$ be a central simple algebra over $L$ whose index is coprime to $p$.
Then there exists an extension $M/K$ of degree coprime to $p$ such that $A \otimes_L(L \otimes_K M)$ is a split algebra. 
\end{prop}

\begin{proof} 
Since $A$ is a central simple algebra over $L$, there exists a finite separable extension $N/L$ such that $A \otimes_L N$ is a split algebra (cf. \cite[Proposition 2.2.3 ]{gisz06}). 
Replacing $N$ by its Galois closure over $K$, we assume that $N/K$ is Galois. 
Let $S_p$ be the $p$-Sylow subgroup of the Galois group of $N/K$ and $M = N^{S_p}$ be the fixed field of $S_p$.
Then $[M : K]$ is coprime to $p$ and $[N : M]$ is a power of $p$. 
Since $[L : K] = p$, $E \otimes_K M \simeq LM \subset N$, and so $[N : L\otimes_K M]$ is a power of $p$. 
Since $A \otimes_LN$ is a split algebra and ind$(A)$ is coprime to $p$, $A \otimes_L(L \otimes_K M)$ is split (cf. \cite[Proposition 13.4.(vi)]{pierce}).
\end{proof}

\begin{thm}
\label{rank3-semisimple} Let $k$ be a global field of characteristic $\neq 2$ and $p$ an odd prime. Let $\ell/k$ be a separable 
field extension of   degree $p$.
Let $K$ be a complete discretely  valued field with residue field $k$ and $L/K$ the unramified extension of degree $p$ with residue field 
$\ell$. Let $G = R_{L/K}(PGL_2)$. Then  there exists a  $G$-torsor $X$ such that $X$ admits a zero-cycle of degree 2 but has no closed point of degree 1 or 2. 
\end{thm}

\begin{proof}  Let   $D$ be a quaternion division algebra  over $L$ as in (\ref{fields}). 
Since $G = R_{L/K}(PGL_2)$,    $H^1(K, G) = H^1(L, PGL_2)$ classifies quaternion algebras over $L$ (cf.  \cite[Theorem 2.4.3 ]{gisz06}), 
 Let $X$ be the $G$-torsor given by the quaternion algebra $D$. 
Let $M/K$ be an extension of degree 1 or 2.
Suppose $X(M) \neq \emptyset$. Then $D \otimes_L(L \otimes_KM)$ is a split algebra (cf. \cite[ Theorem 2.4.3 and 5.2.1]{gisz06}).
In particular, cor$_{L/K}(D) \otimes M$ is trivial  and hence ind(cor$_{L/K}(D))$ is at most 2, leading to a contradiction (\ref{cor}). 
Hence $X$ has no closed point of degree 1 or 2. 

Since $D$ is split over a degree 2 extension of $L$, $X$ has a closed point of degree $2p$. 
By (\ref{split}), there exists a field extension $M/K$ of degree coprime to $p$ such that $D \otimes_L(L \otimes_K M)$ is a split algebra. 
Then $X(M) \neq \emptyset$.
Since cor$_{L/K}(D) \otimes_K M = $ cor$_{L \otimes_K M/M}(D \otimes _L(L \otimes_K M))$ is split and cor$_{L/K}(D)$ has index 4, $[M : K]$ is divisible by 4. 
Since $[M : K]$ is coprime to $p$, $X$ has a zero-cycle of degree $2$. 
\end{proof}

\begin{thm}
\label{rank3-torus} Let $k$ be a global field of characteristic $\neq 2$ and  $p$ an odd prime.
Let $K$ be a complete discretely  valued field with residue field $k$.
Then there exist a torus $T$ over $K$ of rank $p$ and a $T$-torsor $X$ over $K$ such that $X$ has a zero-cycle of degree 2 and has no
 closed point of degree 1 or 2. 
\end{thm}

\begin{proof}  Let $\ell/k$ be a separable field extension of degree $p$ and $L/K$ the unramified extension of degree $p$
with residue field $\ell$. 
 Let   $D$ be a quaternion division algebra  over $L$ as in (\ref{fields}). 
Let $E/L$ be a degree two extension which splits $D$. 
Let $R^1_{E/L} \G_m$ be the kernel of the  morphism $R_{E/L} \G_m \to \G_m$ induced by the norm map and
 let $T  = R_{L/K}(R^1_{E/L} \G_m) $.
Then $T$ is a torus of rank $p$ (cf.  \cite{vosk98}).
We have $H^1(K, T) = H^1(L, R^1_{L/E}\G_m) \simeq L^*/N_{E/L}(E^*)$ (cf. \cite[Lemma 3.2.(a)]{totori}).
Since $L^*/N_{E/L}(E^*)$ is isomorphic to the subgroup of $\Br(L)$ consisting 
of the classes of central simple algebras that split over $E$, 
the quaternion division algebra $D$ defines a class in $H^1(K, T)$. 
Let $X$ be the $T$-torsor corresponding to $D$.
Then as in (\ref{rank3-semisimple}), $X$ admits zero-cycle of degree 2 but has no closed point of degree 1 or 2. 
\end{proof}

\begin{rem}(Colliot-Th\'el\`ene)
\label{colliot} 
Let $G$ be as in (\ref{rank3-semisimple} or \ref{rank3-torus}) with rank of $G$ equal to 3. 
Then by taking $ G \times (SL_2)^r$ or $G \times (\G_m)^r$, one gets examples of 
semisimple groups and   tori   of rank $n$ for every $n \geq 3$ which admit  torsors with   zero-cycles of degree 2 but have no closed points of degree 1 or 2.
\end{rem}

\begin{rem} Since $\Q(t)$ (resp. $\Q_p(t)$) has a discrete valuation $\nu$ with residue field a global field, one can descend
 the above examples
over the completion of $\Q(t)$ (resp. $\Q_p(t)$) at $\nu$ to $\Q(t)$ (resp. $\Q(t)$). 
Thus we have examples of connected linear algebraic groups $G$ over 
$\Q(t)$ (resp. over $\Q_p(t)$) for which the question of Totaro has negative answer. 
\end{rem}
 
\section{An Example over a $p$-adic Field}

In this section, we give an example of a smooth connected linear algebraic group $G$ over a $p$-adic field for which the question of Totaro has a negative answer.

\begin{lem}
\label{ker1} 
Let $K$ be a field and $p$ a prime not equal to char$(K)$.
For $a \in K^*\setminus K^{*p}$, the kernel of the natural homomorphism $K^*/K^{*p} \to K(\sqrt[p]{a})^*/K(\sqrt[p]{a})^{*p}$ is 
generated by the class of $a$. 
\end{lem}

\begin{proof} 
Let $\zeta$ be a primitive $p^{\rm th}$ root of unity in an extension of $K$.
Since $[K(\zeta) : K]$ is coprime to $p$, the map $K^*/K^{*p} \to K(\zeta)^*/K(\zeta)^{*p}$ is injective. 
Thus replacing $K$ by $K(\zeta)$, we assume that $\zeta \in K$.
Then, by Kummer theory, the extension $K(\sqrt[p]{a})/K$ is a cyclic extension with Gal$(K(\sqrt[p]{a})/K)$ generated by $\sigma$ given by $\sigma(\sqrt[p]{a}) = \zeta^i \sqrt[p]{a}$ for some $i$ coprime to $p$. 
Let $b \in K^*$ be such that $b = c^p$ for some $c \in K(\sqrt[p]{a})^{*p}$. 
Then $\sigma(c) = \zeta^j c$ for some $j$.
Since $i$ is coprime to $p$, there exists $i'$ such that $ii' = j$ modulo $p$.
We have $\sigma(\sqrt[p]{b}/\sqrt[p]{a}^{i'}) = \sqrt[p]{b}/\sqrt[p]{a}^{i'}$, hence $\sqrt[p]{b}/\sqrt[p]{a}^{i'} \in K$.
In particular, $b = a^{i'} d^p$ for some $d \in K^*$.
\end{proof}

\begin{lem}
\label{ker} 
Let $K$ be a field and $p$ a prime not equal to char$(K)$.
If $L/K$ is a finite extension of degree $n = p^dm$ such that $p$ does not divide $m$, then the order of the kernel of the natural homomorphism $K^*/K^{*p} \to L^*/L^{*p}$ is at most $p^d$. 
\end{lem}

\begin{proof} 
We prove the lemma by induction on $d$. 
Suppose $d = 0$. 
Then $[L : K]$ is coprime to $p$ and hence the homomorphism $K^*/K^{*p} \to L^*/L^{*p}$ is injective. 
Suppose that $d \geq 1$. 
Suppose there exists $a \in K^* \setminus K^{*p}$ with $a \in L^{*p}$.
Let $E = K(\sqrt[p]{a}) \subseteq L$. 
By (\ref{ker1}),  the kernel of $K^*/K^{*p} \to E^*/E^{*p}$ has order $p$.
Since $[L : E] = p^{d-1}m$, by the induction hypotheses, the kernel of the homomorphism $E^*/E^{*p} \to L^*/L^{*p}$ has order at most $p^{d-1}$. 
Thus the order of $H$ is at most $p^d$. 
\end{proof}
 
\begin{lem}
\label{classes} 
If $k$ is a $p$-adic field, the order of $k^*/k^{*p}$ is at least $p^{d+1}$, where $d = [k : \Q_p]$. 
\end{lem}

\begin{proof} 
Let $q$ be the number of elements  in the residue field of $k$.
By (\cite[Proposition 5.7, p.140]{Neu}), we have $k^* \simeq \Z \oplus (\Z/(q-1)\Z) \oplus (\Z/p^a\Z) \oplus \Z_p^d$ for some $a \geq 0$.
Hence $(\Z/p\Z)^{d + 1}$ is isomorphic to a subgroup of $k^*/k^{*p}$.
\end{proof}

For a $p$-adic field $k$ and $n \geq 1$, there are only finitely many extension of $k$ of degree $n$ (cf. \cite[Section 2.5, Proposition 14]{lang94}), and there exists at least one extension of $k$ degree $n$. 
 
\begin{lem} 
\label{algebras} 
Let $k$ be a $p$-adic field and $E/k$ be a finite extension containing all of the degree $p$ extensions of $k$. 
Let $D$ be a central division algebra of degree $p$ over $E$.
Then for every prime $\ell$, there exists an extension $L/k$ of degree coprime to $\ell$ and a degree $p$ extension $M/L$ such that $D \otimes_E (M \otimes_k E) $ is a split algebra. 
\end{lem}

\begin{proof} 
Let $N = [E : k]$ and $\ell$ a prime. 
For every $i$, $1 \leq i \leq N$, write $i = p^{d_i} m_i$ for some $d_i \geq 0$ and $m_i$ coprime to $p$. 
Let $d$ be a natural number coprime to $\ell$ such that $d \geq \prod (d_i + 1) $ for all $i$, $1 \leq i \leq N$.
Since $k$ is a $p$-adic field, there exists an extension $L/k$ of degree $d$. 
Write $L \otimes E \simeq \prod L_j$. 
Then each $L_j$ is an extension of $L$ of degree at most $N$. 
For each $j$, let $H_j$ be the kernel of the natural homomorphism $L^*/L^{*p} \to L_j^*/L_j^{*p}$. Let $i = [L_j : L]$. 
Since $1 \leq [L_j : L] \leq N$, by ( \ref{ker}), the order of $H_j$ is at most $p^{d_i + 1}$. 
Thus, the order of the product of all $H_j$ is at most $p^d$. 
Since the order of $L^*/L^{*p}$ is at least $p^{d+1}$ (cf. \ref{classes}), there exists $a \in L^*$ such that $a \not\in H_j$ for all $j$. 
In particular $a \not\in L_j^{*p}$ for all $j$. 
Since $L_j$ is a $p$-adic field and $D$ a central simple algebra of degree $p$, $D \otimes_E L_j(\sqrt[p]{a})$ is a split algebra (cf. \cite[Section 6.1, Corollary 1]{CF}). 
Let $M =L(\sqrt[p]{a})$.
Then $M \otimes_k E \simeq \prod L_j(\sqrt[p]{a})$, hence $D \otimes_E( M \otimes_k E)$ is a split algebra. 
\end{proof}

\begin{thm} 
\label{theorem}
Let $k$ be a $p$-adic field and $E/k$ be a finite extension containing all of the degree $p$ extensions of $k$. 
Let $G = R_{E/k}(PGL_p)$ be the linear algebraic group given by the Weil transfer of $PGL_p$ from $E$ to $k$. 
Then every non-trivial principal homogeneous space of $G$ over $k$ admits a zero-cycle of degree $p$ but has no closed point of degree $p$.
\end{thm}

\begin{proof} 
Let $X$ be a principal homogeneous space of $G$ over $k$ with $X(k) = \emptyset$. 
Since $H^1(k, G)$ classifies principal homogenous spaces of $G$ over $k$ (cf. \cite[Section 1.5, Proposition 33]{serre97}), $X$ corresponds to an element of $H^1(k, G)$. 
By \cite[Section 1.5.b]{serre97}, we have $H^1(k, G) = H^1(E, PGL_p)$. 
Since $H^1(E, PGL_p)$ classifies central simple algebras of degree $p$ over $E$ (cf.  \cite[Theorem 2.4.3]{gisz06}), $X$ corresponds to central simple division algebra $D$ of degree $p$ over $E$.
Further, for any extension $F/k$, $X(F) \neq \emptyset$ if and only if $D \otimes_E (F \otimes_kE)$ is split (cf.  \cite[Theorems 2.4.3 and 5.2.1]{gisz06}). 
Let $\ell_1$ be a prime.
By (\ref{algebras}), there exists an extension $L_{\ell_1}/k$ of degree coprime to $\ell_1$ and a degree $p$ extension $M_{\ell_1}/L_{\ell_1}$ such that $D \otimes_E (M_{\ell_1} \otimes_k E)$ is a split algebra.
In particular, $X(M_{\ell_1}) \neq \emptyset$. 
Let $d = [L_{\ell_1} : k]$. 
Then $[M_{\ell_1} : k ] = dp$.
For every prime $\ell$ dividing $d$, by \ref{algebras}, there exists an extension $L_{\ell}/k$ of degree coprime to $\ell$ and a degree $p$ extension $M_{\ell}/L_{\ell}$ such that $D \otimes_E (M_{\ell} \otimes_k E)$ is a split algebra.
In particular, $X(M_\ell) \neq \emptyset$. 
Then the gcd of $[M_{\ell_1} : k]$ and $[M_\ell : k]$ for all $\ell$ dividing $d$ is $p$. 
Thus $X$ admits a zero-cycle of degree $p$.
Let $N/k$ be a degree $p$ extension. 
Then, by the choice of $E$, $N \subseteq E$.
Thus $N \otimes E \simeq \prod E_i$ with $[E_i : E] \leq p-1$. 
Since $D$ is a central division algebra of degree $p$ over $E$, $D \otimes_E E_i$ is division for all $i$. 
Then $D \otimes_k N$ is not split, hence $X(N) = \emptyset$. 
Thus there is no closed point of $X$ of degree $p$. 
\end{proof}

\begin{ex} 
\label{example-lf}
Let $k = \Q_2$ and $E = k(\sqrt{2}, \sqrt{3}, \sqrt{5})$.
Then $[E : k] = 8$ and $E$ contains every quadratic extension of $k$. 
Let $G = R_{E/k}(PGL_2)$. 
Then the rank of $G$ is 8. 
By \ref{theorem}, every non-trivial $G$-torsor admits a zero-cycle of degree 2 but has no closed point of degree 2.
As in (\ref{rank3-torus}), we also get an example of torus $T$ over $k$ of rank 8 and a non-trivial $T$-torsor that admits a zero-cycle of degree 2 but has no closed point of degree 1 or 2. 
\end{ex}

\begin{ex} 
\label{example-gf}
Let $k = \Q$, $E = k(\sqrt{2}, \sqrt{3}, \sqrt{5})$, and $G = R_{E/k}(PGL_2)$. 
Let $D = (a, b)$ be the quaternion division algebra over $E \otimes \Q_2$.
Let $a', b' \in E$ be close to $a$ and $b$, and let $D' = (a', b')$. 
Then $D' \otimes_E(E\otimes \Q_2) \simeq D$.
Let $X$ be the $G$-torsor associated to $D$. 
Since $X$ has no closed point of degree 1 or 2 over $\Q_2$ (\ref{example-lf}), $X$ has no closed point of degree 1 or 2. 
Using Krasner's lemma, we can show that $X$ has a zero-cycle of degree 2.
\end{ex}

\bibliographystyle{alpha}
\bibliography{totaro}

\end{document}